\documentclass[12pt]{amsart}
\usepackage{epsfig,amsmath}
\usepackage{verbatim}
\usepackage[dvipsnames, usenames]{color}
\usepackage[colorlinks=true, pdfstartview=FitV, linkcolor=BrickRed, citecolor=BrickRed, urlcolor=BrickRed]{hyperref}

\numberwithin{equation}{section}

\newtheorem{thm}[subsection]{Theorem}
\newtheorem{lem}[subsection]{Lemma}
\newtheorem{cor}[subsection]{Corollary}
\newtheorem{prop}[subsection]{Proposition}
\newtheorem{obs}[subsection]{Observation}

\theoremstyle{definition}
\newtheorem{definition}[subsubsection]{Definition}
\newtheorem{rem}[subsubsection]{Remark}
\newtheorem{eg}[subsubsection]{Example}

\newcommand{\R}{\mathbb{R}}
\newcommand{\Z}{\mathbb{Z}}
\newcommand{\C}{\mathbb{C}}
\newcommand{\US}{\mathbf{S}}

\def\eob{\hfill\qedsymbol}

\def\cmc{\textsc{cmc}}

\def\b{\beta}
\def\g{\gamma}
\def\G{\Gamma}
\def\i{\mathrm{i}}

\def\th{\theta}
\def\Th{\Theta}
\def\eps{\varepsilon}
\def\ff{\phi} % flux functional

\def\ggc{$\gg$-crowded}

\def\ds{\displaystyle}
\def\div#1{\mathrm{div_{\Sigma}\left(#1\right)}}
\def\Div#1{\mathrm{div_{N}\left(#1\right)}}

\def\O{\mathcal{O}}
\def\cy{\mathcal{Z}}
\def\bo{\mathcal{B}}
\def\cs{\mathcal{C}}
\def\rc{\mathcal{K}}
\def\ch{\mathcal{S}}
\def\gn{\mathcal I}
\def\gg{\mathcal G}
\def\gs{\mathcal L}
\def\lgn{L(\mathcal{I})}

\def\lgs{L(\mathcal{L})}
\def\c{\mathbf{k}}
\def\bd{\partial}
\def\lra{\longrightarrow}

\def\mv#1{\left|#1\right|_{\mu}}
\def\and{\quad\text{and}\quad}

\def\l{\left}
\def\r{\right}

\catcode`@=11 \@mparswitchfalse  %This puts the note in the right margin.

\newcounter{mnotecount}[page]

\begin{document}

\parskip 6pt
\parindent 0pt
\baselineskip 15pt

\title[CMC, flux conservation \& symmetry]{Constant mean curvature, flux conservation, and symmetry}

\author{Nick Edelen \& Bruce Solomon}
\address{
	Stanford University, Stanford CA (Edelen) \hfill\break 
	Indiana University, Bloomington, IN (Solomon)
}

\date{Begun August, 2011. Latest draft \today.}

\begin{abstract}
	As first noted in \cite{kks}, constant mean curvature implies a homological conservation 
	law for hypersurfaces in ambient spaces with Killing fields.
	In Theorem \ref{thm:law} here, we generalize that law by relaxing the topological 
	restrictions assumed in \cite{kks},
	and by allowing a weighted mean curvature functional. 
	We also prove a partial converse (Theorem \ref{thm:converse}) which roughly
	says that when flux is conserved 
	along a Killing field, a hypersurface splits into two regions: one 
	with constant (weighted) mean curvature, and one preserved by the Killing field. 
	We demonstrate our theory by using it to derive a first integral for 
	helicoidal surfaces of constant mean curvature in $\,\R^{3}$, i.e., ``twizzlers.''
\end{abstract} 

\maketitle

%%%%%
\section{Introduction}
		
	Constant mean curvature (``\cmc'') imposes a homological flux conservation law on
	hypersurfaces in ambient spaces with non-trivial Killing fields. 
	This was first observed and exploited by Korevaar, Kusner, \& Solomon in their 1989 paper on 
	the structure of embedded \cmc\ surfaces in $\,\R^{3}\,$ \cite{kks} (see \cite{k} for an
	alternative exposition). In Theorem \ref{thm:law} here, 
	we generalize that law by relaxing the topological restrictions assumed by \cite{kks},
	and by allowing a weighted version of the mean curvature functional. 
	We further extend the theory via Theorem \ref{thm:converse}, which gives a partial converse
	to the conservation law. Roughly, it states that when the appropriate flux is conserved 
	along enough Killing fields, the hypersurface splits into two regions (though either may
	be empty): a region with constant (weighted) mean curvature, and a region preserved by the 
	Killing fields.
		
	We apply our results in Section \ref{eg:twiz3} by using them to quickly derive the seemingly 
	\emph{ad hoc} {first integral} 
	that Perdomo \cite{p1}, DoCarmo \& Dajczer \cite{dcd}, and others have used to analyze the moduli
	space of \cmc\  surfaces with helicoidal symmetry, also known as 
	\emph{twizzlers}\footnote{Twizzlers have also been studied by W\"underlich \cite{w} and, 
	more recently, Halldorsson \cite{h}.}.
	In general, constancy of weighted mean curvature
	is characterized by a non-linear second-order PDE, and its N\"oetherian reduction 
	to a first-order condition makes it easier to analyze.
	
	When a \cmc\ hypersurface $\,\Sigma\,$ in a manifold $\,N\,$ is preserved by the action of a continuous
	isometry group $\,\gg$, one can project it into the orbit space $\,N/\gg$.  
	The {projected} hypersurface $\,\Sigma/\gg\,$ will then be stationary for the \emph{weighted} 
	functional introduced in \S\ref{ssec:fv}. We analyze the weighted functional and the resulting
	weighted mean-curvature invariant with an eye toward this fact. We suspect that virtually all we 
	do here could be developed in a more general, stratified context encompassing both riemannian 
	manifolds, and their quotients under smooth group actions. 
	
	We stick with smooth ambient 
	manifolds here, but the orbit space viewpoint can be helpful, and our case study in 
	\S\ref{eg:twiz3} could easily have been carried out in that setting.
	The approach we demonstrate there can also be adapted to spherical and hyperbolic 
	space forms. The first author's report \cite{e} sketches out one way to do that, but 
	we describe the orbit-space approach to those examples in our final Remark 
	\ref{rem:etc}.\bigskip
		
%%%%%
\section{Preliminaries}
	Let $\,N\,$ denote an $n$-dimensional oriented riemannian manifold, and consider a smooth, connected, 
	oriented, properly immersed hypersurface $\,f:\Sigma^{n-1}\to N$.  We will feel free to
	write $\,\Sigma\,$ when we mean $\,f(\Sigma)\,$ or even $\,f:\Sigma\to N$, leaving context
	to clarify our intentions.
	
	Let $\,\nu\,$ denote the unit normal that completes the orientation of $\,\Sigma\,$ to that of 
	$\,N$. The \emph{mean curvature function} $\,h:\Sigma\to\R\,$ is the trace of the shape
	operator $\nabla\nu$. Notationally,
	\begin{equation}\label{eqn:h}
		h = \div{\nu}
	\end{equation}
	Here $\,\mathrm{div}_{\Sigma}(Y)\,$ denotes the \emph{intrinsic divergence} of a vectorfield 
	$\,Y\,$ along $\,\Sigma$, that is, the trace of the endomorphism $\,T\Sigma\to T\Sigma\,$ 
	gotten at each $\,p\in \Sigma\,$ by projecting the ambient covariant derivative $\,\nabla Y\,$ 
	onto $\,T_{p}\Sigma$. One may compute $\,\mathrm{div}_{\Sigma}\,$ locally using any orthonormal basis 
	$\,\{e_{i}\}\,$ for $\,T_{p}\Sigma\,$ via
	\[
		\div Y:= \sum_{i=1}^{n-1}\nabla _{e_{i}}Y\cdot\,e_{i}\ .
	\]
	%%%%%
	\subsection{Chains and $k$-area.}\label{ssec:chains}
	The homology of the sequence
	\[
		N\lra(N,\Sigma)\lra\Sigma
	\]
	will play a role below in a way that makes it problematic to work solely with
	smooth submanifolds. We therefore work with a class of piecewise smooth objects:
	
	\begin{definition}
		A \textbf{smooth $r$-chain} (or simply \textbf{chain}) in a smooth 
		manifold $\,M\,$ is a finite union of smoothly immersed oriented $r$-dimen\-sional
		simplices. We regard a chain $X$ as a formal homological sum:
		\begin{equation}\label{eqn:chain}
			X = \sum_{i=1}^{m} m_{i}f_{i}
		\end{equation}
		Here each $\,f_{i}:\Delta\to M$, immerses the standard closed, oriented $r$-simplex 
		$\,\Delta\,$ (along with its boundary) smoothly into $\,M$.  The $m_{i}$'s are (for us)
		always integers.
		
		We denote the \textbf{support} of a chain $\,X\,$ by $\,\|X\|$. 
		
		We write $\ch_r(M)$ for the \textbf{group of smooth $r$-chains in} $M$, and $\,\bd X\,$ 
		for the homological \textbf{boundary} of a chain $\,X$, while $\cy_{i}(M)\,$ 
		and $\,\bo_{i}(M)\,$ denote the spaces of $i$-dimensional \textbf{cycles} and
		\textbf{boundaries} (kernel and image of $\,\bd\,$) in $\,M\,$ respectively. 
		Likewise $\,\cy_{i}(M,A)\,$ and $\,\bo_{i}(M,A)\,$ indicate spaces of cycles and 
		boundaries modulo a subset $\,A\subset M$. 
	\eob
	\end{definition}
	
	Integration of an $r$-form $\,\phi\,$ on $\,M\,$ over such a chain is trivial:
	\[
		\int_{X}\phi := \sum_{i=1}^{m}m_{i}\int_{\Delta}f^{*}_{i}\phi
	\]
	where $\,f_{i}^{*}\,$ denotes the usual pullback.
	
	Given a riemannian metric on $\,M$,
	one can also integrate \textit{functions} over chains, and most importantly for our purposes, compute
	weighted volumes.
	
	\begin{definition}\label{def:mass}
		Let $\,\mu:M\to\R\,$ be any continuous function.
		Define the \textbf{$\mu$-weighted $r$-volume} $\,\mv{X}\,$ of the $r$-chain $\,X\,$ 
		in (\ref{eqn:chain}) as
		
		\begin{eqnarray*}
			\mv{X}
			&:=&
			\sup\left\{
				\int_{X}e^{\mu}\phi
				\,\colon\, 
				\phi\ \text{is an $r$-form on $M$ with $\left\|\phi\right\|_{\infty}\le 1$}
			\right\}
		\end{eqnarray*}
	\eob
	\end{definition}
	
	For a single immersed simplex, the usual riemannian volume integral gives a simpler
	definition. To allow coincident, oppositely oriented simplices to cancel, however,
	we need the definition above\footnote{Definition \ref{def:mass} amounts to 
	a weighted version of the \emph{mass} of $\,X\,$ as a \emph{current}, in the sense of 
	geometric measure theory \cite[p.~358]{gmt}.}.
	
	Finally, note that because Stokes' Theorem 
	holds for immersed $r$-simplices, it holds for $r$-chains as well.
	
%%%%%
\subsection{Symmetry.}\label{ssec:sym}
	Our work here is vacuous unless the ambient space $\,N\,$ has non-trivial Killing fields.
	
	Write $\,\gn\,$ and $\,\lgn\,$ respectively for the isometry group of $\,N\,$ and 
	its Lie algebra. 
	Identify $\,\lgn\,$ with the linear space of Killing fields on $N$ in the usual way,
	associating each $\,Y\in\lgn\,$ with the Killing field (also called $\,Y$) we get 
	by differentiating the 	flow that sends $\,p\in N\,$ along the path $\,t\mapsto\exp(tY)\,p$.  
	We write $\,Y_{p}\,$ for the value of $\,Y\,$ at $\,p$. 
	
	One often studies \cmc\ hypersurfaces (like surfaces of revolution and twizzlers in 
	$\R^{3}$) in relation to the action of a closed, connected subgroup $\,\gs\subset\gn$.  
	Though it complicates our exposition to some extent, the presence of 
	such a subgroup $\,\gs\,$---like that of the density function $\,e^{\mu}\,$---lets
	us broaden our theory. Even when $\,\mu\equiv 0\,$ and $\,\gs\,$ is the full 
	isometry group of $\,N$, however, our results go beyond those of \cite{kks}.
		
	In Theorem \ref{thm:converse} (a converse to our Conservation Law)
	we must consider the possibility that all Killing fields associated with $\,\gs\,$ 
	lie tangent to an open subset $\,S\,$ of our hypersurface 
	$\,\Sigma\subset N$.  The following Lemma (and its Corollary) then lets us 
	deduce $\gs$-invariance of $\,S$. 
	
	\begin{lem}\label{lem:nbd}
		Suppose $\,S\subset N\,$ is a hypersurface, and that for some $\,Y\in\lgs$, 
		we have $\,Y_{p}\in T_{p}S\,$ for every $\,p\in S$.  Then for each $\,p\in S\,$, there
		exists a compact neighborhood $\,\O_{p}\subset S\,$ and an $\,\eps>0\,$ such that
		such that $\,e^{tY}q\in S\,$ whenever $\,|t|<\eps\,$ and $\,q\in\O_{p}$.
	\end{lem}
	
	\begin{proof}
		Since $\,S\,$ is a submanifold, some open set $\,W\subset N\,$ 
		contains $\,S$, but no point of  $\,\bar S\setminus S\,$ ($\bar S$ = closure of $S$). 
		Let $\,\Th:S\times\R\to N\,$ denote the flow of $\,Y\,$, so that $\,\Th(q,t):=\exp(tY)q$.
		Then $\,\Th^{-1}(W)\,$ is an open neighborhood of $\,S\times \{0\}$. 
		
		Now $\,\Th(q,t)$ parametrizes the integral curve of $\,Y\,$ with
		initial velocity $\,Y_{q}$.  But $\,Y_{p}\in T_{p}S\,$ for all $\,p\in S$, 
		and first-order ODE's have unique solutions, so this curve must stay in $S$ for all 
		$\,(q,t)\in W$.  It follows that $\,\Th^{-1}(W)=\Th^{-1}(S)$. 
		
		For any compact neighborhood  $\,\O_{p}\,$ of $\,p\in S$, there now 
		exists an $\,\eps>0\,$ such that
		 \[
		 	\O_{p}\times(-\eps,\eps)\subset \Th^{-1}(S)
		\]
		
		and the Lemma consequently holds with this choice of $\,\O_{p}$ and $\,\eps$. 
	\end{proof}\medskip

%	\begin{cor}\label{cor:orbit} 
%		If $\,S\subset N\,$ is a hypersurface, with $\,Y_{p}\in T_{p}S\,$ for each $\,p\in S\,$
%		and all $\,Y\in\lgs$, then for any orbit $\,Q\,$ of the $\,\gs$-action on $\,N$,
%		$\,Q\cap S\,$ will be open in $\,Q$. 
%	\end{cor}
%	
%	\begin{proof}
%		If $\,p\in Q\cap S$, the Lemma yields $\,g\,p\in S\,$ for all $\,g\in 
%		\mathcal{O}_{e}$, and the set of all such $\,g\,p\,$ clearly cover a neighborhood of 
%		$\,p\,$ in $\,Q$. 
%	\end{proof}
%	\bigskip
	
%%%%%
\subsection{Flux}\label{ssec:flux}
	Korevaar, Kusner, and Solomon showed in \cite{kks} that when a
	hypersurface $\,\Sigma\subset N^{n}\,$ has {constant} mean curvature $\,h\equiv H$, 
	and the homology groups $\,H_{n-1}(N)\,$ and $\,H_{n-2}(N)\,$ are both trivial (over 
	$\Z\,$ --- all homology groups in this paper have integer coefficients), there exists a \emph{flux 
	homomorphism}
	\[
		\phi:H_{n-2}(\Sigma)\otimes\lgn\to \R
	\]
	
	defined by assigning, to any Killing field $\,Y\,$ and any class 
	$\,\c\in H_{n-2}(\Sigma)$, \textbf{the flux $\,\phi(\c,Y)\,$ of $\,Y\,$ across $\c$}, where
	
	\begin{equation}\label{eqn:kks}
		\phi(\c,Y):=\int_{\G}\eta\cdot Y+ H\int_{K}\nu\cdot Y
	\end{equation}
	Here 
	\begin{itemize}
	
		\item
		$\G\,$ can be an $(n-2)$-cycle representing $\,\c\,$ \medskip
		
		\item
		$K\subset N\,$ can be any $(n-1)$-chain bounded by $\,\G\,$\medskip
		
		\item 
		$\eta\,$ is the orienting unit conormal to $\,\G\,$ in $\,\Sigma$, and \medskip
		
		\item
		$\nu\,$ is the orienting unit normal to $\,K\,$ in $\,N$.  
	\end{itemize}
	
	To ensure that $\,\phi(\c,Y)\,$ is well-defined by \eqref{eqn:kks}, 
	\cite{kks} makes two topological assumptions: namely, \emph{that $\,H_{n-1}(N)\,$ and $\,H_{n-2}(N)\,$ 
	both vanish}. 
	The vanishing of $\,H_{n-2}(N)\,$ ensures $\,\G\,$ will 
	bound \emph{some} chain $\,K$, while that of $\,H_{n-1}(N)\,$ means any competing chain 
	$\,K'\,$ with $\,\bd K'=\bd K\,$ can be written $\,K'=K+\bd U\,$ for some $n$-chain $\,U$.  
	Since Killing fields are divergence-free, the Divergence Theorem then makes the second integral 
	in (\ref{eqn:kks}) independent of the choice of $\,K$. 
	
	Here, we extend this \cite{kks} flux theory in several ways. 
	
	First, in \S\ref{ssec:fv}, we broaden the mean curvature functional by allowing $\mu$-weighted
	area and volume as in Definition  \ref{def:mass}.
	This is a minor tweak of the standard theory, but it does \emph{not}  correspond
	to a mere conformal change of metric, since $n$- and $(n-1)$-dimensional volume scale
	differently under conformal change. We do this with a geometric 
	application in mind: the $\mu$-weighted theory relates the geometry of $\gs$-invariant 
	\cmc\ hypersurfaces in $\,N\,$ to that of hypersurfaces in the orbit space $\,N/\gs\,$ 
	(see Remark \ref{rem:mu} below).
	
Second, and more importantly, we eliminate the homological triviality assumptions mentioned above.  
Though we follow the same variational stragety as in KKS, we show the flux invariant lives more naturally in a certain \emph{relative} homology group.  Instead of focusing the invariant on $(n-2)$-cycles in the surface $\Sigma$, we realize the flux as an invariant on certain $(n-1)$-dimensional relative cycles
we shall call \emph{caps}.

The homological restriction can be naively avoided by defining the flux on $H_{n-1}(N, \Sigma)$.  When
$H_{n-2}(\Sigma)\ne 0$, however, one gets a more sensitive invariant by designating a set of ``reference cycles." We call this set a \emph{spine}.  It not only gives better invariants; it tends to make
flux calculations more tractable.

The new viewpoint reproduces the \cite{kks} invariant when $H_{n-1}(N) = H_{n-2}(N) = 0$.  
In that case, the reference cycle is trivial, and the long exact sequence for the pair $(N, \Sigma)$,
namely\smallskip
\[
	0=H_{n-1}\l(N\r)\lra H_{n-1}\l(N,\Sigma\r)\stackrel{\bd}{\lra} H_{n-2}\l(\Sigma\r)\lra H_{n-2}\l(N\r)= 0
\]
shows that $H_{n-1}(N, \Sigma) \cong H_{n-2}(\Sigma)$.
	
	We derive our generalized conservation law in \S\ref{sec:conslaw}, and then, 
	in \S \ref{sec:converse}, develop a partial converse.
	Before proceeding to these extensions however, we present a motivating example that
	we can review later as an illustration of our theory.
	
	\begin{eg}\label{eg:twiz1}
	\emph{Twizzlers} are ``helicoidal''\cmc\ surfaces invariant 
	under a 1-parameter group of screw motions in $\,\R^{3}$. Any such surface can be gotten
	by applying a screw motion to a curve $\,\g\,$ in a plane perpendicular to the screw-axis.
	The resulting helicoidal surface will then have mean curvature $\,h\equiv H\,$ if and only 
	if $\,\g\,$ satisfies an easily-derived second order ODE. As others 
	(\cite{dcd}, \cite{w}, \cite{p1}, \cite{h}) have noted, however, the second order ODE has a useful 
	first integral. We show how to derive it from flux conservation below.
	
	The conservation law formulated in \cite{kks}, however, yields nothing for twizzlers, since
	the typical \cmc\ twizzler is generated by a non-periodic curve $\,\g\,$ in the 
	transverse plane, and thus lacks homology. 
	To remedy that, one can mod out the translational period of the helicoidal motion, realizing
	the twizzler as an immersion of a \emph{cylinder} in 
	$\,N:=\R^{2}\times \US^{1}$.  
	Cylinders \emph{do} have non-trivial loops, but those loops don't bound in $\,N$, and hence
	can't be capped off as required by \cite{kks}. 
	
	Our approach evades that obstruction; see Example \ref{eg:twiz2} and 
	\S\ref{eg:twiz3}.
\eob
\end{eg}
	
%%%%%
\section{Conservation}\label{sec:conslaw}

	Like \cite{kks}, we derive flux-conservation using a constrained first-vari\-ation formula. 
	We make two notable modifications, however.
	
	First, we {weight} both the areas of hypersurfaces and the volumes
	of domains by an $\gs$-invariant \emph{density function}
	\begin{equation}\label{eqn:density}
		e^{\mu}: N\lra (0,\infty)
	\end{equation}
	Here $\,\mu\,$ can be any smooth function fixed by $\,\gs$.  The formula in 
	\cite{kks} effectively takes $\,\mu\equiv 0$, as will
	become clear in \S\ref{ssec:fv} below.

Secondly, we encode the homology of our immersion $f:\Sigma\to N$ into a set of reference cycles $B$.  Let $f_*$ denote the induced homomorphism
\[
f_* : H_{n-2}(\Sigma) \to H_{n-2}(N)
\]

\begin{definition}[Spine]\label{def:spine}
	We call a subgroup $B \subset \cy_{n-2}(N)$ %(note: not $\,H_{n-1}(N)$) 
	a \emph{spine} for the pair $(N, \Sigma)$ if: 
	\begin{itemize}
	 \item[a)] $B \cap \cy_{n-2}(\Sigma) = 0$
	 \item[b)] $B$ generates $f_* H_{n-2}(\Sigma)$
	 \item[c)] the composition $B \to \cy_{n-2}(N) \to H_{n-2}(N)$ is injective
	\end{itemize}
	
	We won't always draw an explicit distinction between the subgroup $B$ and a set of generating cycles 
	for $B$.
\eob\end{definition}

A non-trivial spine lets us assign fluxes to classes in 
$H_{n-2}(\Sigma)$ that don't bound in $\,N$.

Note that a spine for $\,(N,\Sigma)\,$ always exists. 
Indeed, any independent set of cycles that generate $\,f_{*}H_{n-2}(\Sigma)\,$ in 
$\,H_{n-2}(N)\,$ will satisfy conditions (b) and (c) of Definition \ref{def:spine}, and one can 
always perturb slightly, if needed, to realize (a). That condition is really an artifact of the language we use to define the flux invariants; we want the sum $\cy_{n-2}(\Sigma) + B$ to be direct. The assumption 
could be omitted in favor of more precision in distinguishing "caps with non-trivial spines" and "caps without spines."

\begin{definition}[Cap]\label{def:cap}
A \emph{cap} $K$ is any chain in $\ch_{n-1}(N)$ such that
\[
	\bd K \in \cy_{n-2}(\Sigma) \oplus B% = \cy_{n-2}(\Sigma) + B
\]
As the kernel of the composition
\[
	\ch_{n-1}(N) \stackrel{\bd}{\lra} \ch_{n-2}(N) \lra \ch_{n-2}(N) / \l(\ch_{n-2}(\Sigma) \oplus B\r)
\]
the set of all caps forms a group, which we denote by $\cy(N, \Sigma, B)$.

A \emph{reduced cap} is a class belonging to the quotient
\[
	\rc(N, \Sigma, B) = \cy(N, \Sigma, B) / \bo_{n-1}(N, \Sigma)
\]
\eob
\end{definition}

We call two caps $\,K$, $K'\,$ \textit{homologous}, written $\,K \sim K'$, if they represent the 
same reduced cap in $\rc(N, \Sigma, B)$.

In spirit, a reduced cap is a class in $H_{n-1}(N, \Sigma \cup \|B\|)$, where $\,\|B\|\,$ denotes
the support of $\,B$.  Indeed, when $\|B\|\,$ is disjoint from $\Sigma$, we have
$\,\rc(N, \Sigma, B)=H_{n-1}(N, \Sigma \cup \|B\|)$.   When $\|B\|$ does meet $\Sigma$, however, ambiguity 
can arise as to which part of $\partial K$ to take as $\,\b\,$ in Observation \ref{obs:cap} below. 
The need to remove that ambiguity motivated Definition \ref{def:cap}. The direct sum decomposition 
of $\,\cy(N,\Sigma,B)\,$ there immediately yields the fact we need:

\begin{obs}\label{obs:cap}
For any cap $K$, there exists a unique $\beta \in B$ with $\|\partial K - \beta\|\subset\Sigma$.
\end{obs}

	To make the notions of \emph{spine} and \emph{cap} more concrete, 
	we illustrate using twizzlers:
	
	\begin{eg}\label{eg:twiz2}
		As explained in Example \ref{eg:twiz1}, we may regard a twizzler as a 
		cylinder $\,\Sigma=\R\times\US^{1}\,$ immersed in $\,N=\C\times\US^{1}\,$ 
		and preserved by the helical $\,\US^{1}\,$ action
		\begin{equation}\label{eqn:screw}
			[e^{\i\th}]\left(z,\,e^{\i t}\right) = \left(e^{\i\th}z,\,e^{\i(t+\th)}\right)
		\end{equation}
		
		The length of the $\,\US^{1}\,$ factor is geometrically significant,  
		but can take it to be the usual $2\pi$ for purposes of this example.
		
		We call the orbits of the screw-action \emph{helices}. By construction, 
		both $\,N\,$ and the twizzler $\,f(\Sigma)\,$ are foliated by such helices, 
		any one of which generates $\,f_{*}H_{1}(\Sigma)=H_{1}(N)$.  It follows that any helix, 
		viewed as a $1$-cycle in $N$,
		qualifies as a spine for $\,(N,\Sigma)$.  We take the shortest one, namely 
		$\,{\mathbf{0}}\times\US^{1}\subset N$, as our spine $B$. 
		
		Suppose a twizzler is generated by a particular curve $\,\g:\R\to\C$, so that we can
		immerse it in $\,N\,$ via
		\[
			f\left(t,e^{\i\th}\right)= \left(e^{\i\th}\g(t),\,e^{\i\th}\right)
		\]
		For each fixed $\,t\in\R$, the helix $\,\G_{t}:=f(t,\US^{1})\,$ forms a non-trivial
		cycle in $\,H_{1}(\Sigma)$.  Any oriented surface that realizes the homology between
		$\,\G_{t}\,$ and the compatibly oriented cycle $\,\b\in B\,$ is then a \emph{cap} for 
		$\,\G_{t}$. 
		
		For instance, the line segment (or any arc) joining $\,\mathbf{0}\,$ to $\,\g(t)\,$ in 
		$\,\C\,$ will, under the $\,\US^{1}\,$ action \ref{eqn:screw}, sweep out a cap, and all
		arcs give rise to the same reduced cap in this way. Such caps are also preserved by the 
		$\,\US^{1}\,$ action, a useful property that many other caps lack.
	\eob
	\end{eg}

\subsection{First variation.}\label{ssec:fv}
To prepare for our first variation formula, fix a spine $B$ for $(N, \Sigma)$, and suppose we have 
homologous caps $K$, $K'$ in $\cy(N, \Sigma, B)$.  There then exists an $n$-chain $U$ satisfying
\begin{equation}\label{eqn:UKKS}
	\partial U = S+ K - K'
\end{equation}
for some $\,S\in\ch_{n-1}(\Sigma)$.  
Applying the boundary operator to \eqref{eqn:UKKS}, we 
then get
\begin{equation}\label{eqn:bdUKKS}
	\partial K - \partial K'= -\partial S .
\end{equation}
In particular, $\partial K - \partial K'$ is a cycle in $\Sigma$,
and by definition of \textit{cap}, there now exist unique $\beta, \beta' \in B$ such that
\[
	\partial K - \beta, \,\, \partial K' - \beta' \in \cy_{n-2}(\Sigma)
\]
and \eqref{eqn:bdUKKS} forces $\beta = \beta'$. This proves

\begin{prop}\label{prop:bdcap}
	If two caps $K, K'\in\cy(N,\Sigma,B)$ are homologous, there exists a unique $\beta \in B$ so that both $
	\partial K - \beta$ and $\partial K' - \beta$ are supported in $\Sigma$.
\end{prop}

\begin{definition}\label{def:assocspine}
	The Proposition above lets us define the
	\textbf{spine of a reduced cap $\c \in \rc(N, \Sigma, B)$} as the unique 
	$\beta \in B$ with $\bd K - \beta \in \cy_{n-2}(\Sigma)$ for any representative $K$.
\eob
\end{definition}

	In the situation just described, and in the presence of a density
	function $\,e^{\mu}$, we now consider the $n$- and $(n-1)$-dimensional $\mu$-weighted volumes
	$\,\mv{U}\,$ and $\,\mv{S}\,$ of the chains $\,U\,$ and $\,S\,$ respectively 
	(Definition \ref{def:mass}) as we deform along the flow of a smooth vectorfield $Y$.  
	Fix a scalar $H$, and consider the initial derivative of $\,\mv{S}-H\,\mv{U}\,$ with respect to
	this flow, written 
	\begin{equation}\label{eqn:vcfvf}
		\delta_{Y}\left(\mv{S}-H\mv{U}\right)
	\end{equation}
	Calling this the \textbf{($\mu$-weighted) volume-constrained first-varia\-tion} of $S$,
	we obtain our conservation law for hypersurfaces with constant \textit{$\mu$-mean curvature} 
	$\,h_{\mu}\equiv H\,$ as defined in (\ref{eqn:hmu}) below, by evaluating (\ref{eqn:vcfvf}) on 
	Killing vectorfields of $\,N$. To simplify the task, we analyze $\,\delta_{Y}\mv{U}\,$ and 
	$\,\delta_{Y}\mv{S}\,$ separately before combining results.
	
	A familiar derivation shows $\,\delta_{Y}\mv{U}\,$ to equal the integral of $\,\Div Y$ 
	over $\,U$ when $\,\mu\equiv 0$.  A routine modification of that calculation 
	shows that for general $\,\mu$, we get 
	\begin{equation*}
		\delta_{Y}\mv{U}
		=
		\int_{U}\Div{e^{\mu}\,Y}
		=
		\int_{\bd U} e^{\mu}\,Y\cdot \nu
	\end{equation*}
	where $\,\nu\,$ denotes the orienting unit normal along $\,\bd U$. 
	By \eqref{eqn:UKKS} we can rewrite this as
	\begin{equation}\label{eqn:fvu}
		\delta_{Y}\mv{U}
		=
		 \int_{S} e^{\mu}\,{Y} \cdot \nu + \int_{K-K'} e^{\mu}\,{Y} \cdot \nu 
	\end{equation}
	
	A similar modification of the $\mu\equiv 0\,$ case as analyzed in \cite[pp. 46--51]{lms},
	computes the $\mu$-weighted first-variation of $\mv{S}$ along ${Y}$:
	\begin{equation}\label{eqn:fvs1}
		\delta_Y\mv{S} 
		= 
		\int_S e^{\mu}\,d\mu(\nu)\nu\cdot Y 
			+ \div{e^{\mu}\,Y^{\top}}
			+ \div{e^{\mu}\,Y^{\bot}}
	\end{equation}
	Here $\,Y^{\top}\,$ and $\,Y^{\bot}\,$ signify the tangential and normal components, 
	respectively, of $\,Y\,$ along $\,S$.  
	
	Recall that for vectorfields \emph{tangent} to $\,\Sigma$, the Divergence Theorem applies in 
	its usual form: Given an $(n-1)$-chain $\,S\,$ in $\,\Sigma\,$ with oriented unit 
	conormal $\,\eta\,$ along its boundary, we have
	\[
		\int_{S}\div{X}\ = \int_{\bd S}X\cdot\eta
		\qquad\text{($X$ tangent to $\Sigma$)}
	\]	
	
	For vectorfields \emph{normal} to $\,\Sigma$, on the other hand, the divergence operator
	invokes the mean curvature of $\,\Sigma$, due to \eqref{eqn:h}. 
	When $\,Z = (Z\cdot\nu)\nu\,$ is purely normal, then, the Leibniz rule yields\medskip	
	\[
		\int_{S}\div{Z}\ = \int_{S}(Z\cdot\nu)\,h
		\qquad\text{($Z$ normal to $\Sigma$)}
	\]
	
	Accordingly, we define the \textbf{$\mu$-mean curvature $\,h_{\mu}\,$} along $\,\Sigma\,$ 
	as\smallskip
	\begin{equation}\label{eqn:hmu}
		h_{\mu}:= h+ d\mu(\nu)\ .
	\end{equation}
	
	Using this notation, the facts above reduce (\ref{eqn:fvs1}) to\medskip	
	\begin{equation}\label{eqn:fvs}
		\delta_Y\mv{S}
		=
		\int_{\bd S} e^{\mu}\,{Y} \cdot \eta\ \  + \ \int_S e^{\mu}\,h_{\mu}\,Y\cdot\nu 
	\end{equation}
	
	Finally, using \eqref{eqn:fvu}, \eqref{eqn:fvs}, and \eqref{eqn:bdUKKS}, we can put our 
	{volume-constrained first-variation formula} (\ref{eqn:vcfvf}) into the form we need:
	
	\begin{equation}\label{eqn:vcfvf2}
		\begin{array}{rcl}
			\lefteqn{\delta_Y(\mv{S} - H\mv{U})}\\
			&&\\
			&=&
			\ds-\int_{\partial K - \partial K'}e^{\mu}\,\eta\cdot Y\ -H\int_{K-K'}e^{\mu}\,\nu\cdot Y \\
			&&\\ 
			&&\qquad 
			+ \ds\int_{S}e^{\mu}\,\left(h_{\mu}-H\right)\,\nu\cdot Y
		\end{array}
	\end{equation}\smallskip

	\begin{rem}\label{rem:mu}
		The $\mu$-mean curvature $\,h_{\mu}\,$ arises naturally in the context of 
		riemannian submersions, which we encounter here whenever a compact
		Lie group $\,\gg\,$ of dimension $\,k>0\,$ acts isometrically on a riemannian manifold
		$\,X$.  In that situation, the principal orbits (roughly speaking, the orbits of highest
		dimension) foliate a dense open subset $\,X'\subset X$, and the submersion 
		$\,X'\to X'/\gg\,$ becomes riemannian, given the right metric on 
		$\,X'/\gg\,$ (cf. \cite{hl}).
		
		In any case, every riemannian submersion $\,\pi:P\to N\,$ induces a 
		\textbf{fiber volume function}\smallskip
		\[
			e^{\mu}:N\to(0,\infty),\qquad e^{\mu}(p):= \left|\pi^{-1}(p)\right|
		\]
		
		where $\,|\pi^{-1}(p)|\,$ is the $k$-dimensional volume of the fiber over 
		$\,p$.  A standard first-variation calculation then shows:
		
		\begin{obs}
			The $\mu$-mean curvature $h_{\mu}$ of a hypersurface $\,\Sigma\subset N\,$
			gives the classical mean curvature $h$ of its preimage $\,\pi^{-1}(\Sigma)\subset P$. 
		\end{obs}
		
		In the context of an isometric $\gg$-action as discussed above, one may then study 
		$G$-invariant hypersurfaces of constant (classical) mean curvature 
		$\,h\equiv H\,$ in $\,X\,$ by considering, instead, hypersurfaces of constant $\mu$-mean
		curvature $\,h_{\mu}\equiv H\,$ in the orbit space $\,X/\gg$.  This can be especially fruitful 
		when $\,X/\gg\,$ is just two- or three-dimensional. We consider examples 
		involving twizzlers at the end of the paper.
%		The requirement that $G$ act freely on $\,P\,$ seems too restrictive to us. The orbit
%		volume function makes sense for any smooth action of a compact group $\,G\,$ on a riemannian
%		manifold $\,P$.  The orbits of highest dimension then foliate an open dense subset $\,U\subset P\,$
%		(see \cite{msy}),  and the notion of $\mu$-mean curvature then extends to $\,U/G\,$ just as it does
%		on the base-space of a principal bundle. 
%		Technicalities arise because the full orbit space $\,P/G\,$ typically has (topological) 
%		boundary strata of various dimensions, but we think
%		that $\mu$-mean curvature could be a useful concept in that setting too.
	\eob
	\end{rem}
	
	In any case, the constrained first-variation formula \eqref{eqn:vcfvf2} 
	lets us extend the conservation law presented in \cite{kks}. 
	As before, $\,\gs\subset\gn\,$ denotes a $\,\mu$-preserving
	group of isometries on $\,N$, and the Killing fields that generate 
	its identity component correspond to $\,\lgs$.
	
	\begin{thm}[Conservation law]\label{thm:law} 
		Suppose $\Sigma\subset N$ is an oriented hypersurface with $h_{\mu}\equiv H$, and 
		$B$ is a spine for the pair $(N,\Sigma)$.  Then the formula
		\begin{equation}\label{eqn:flux}
			\ff_{B}[\c](Y) := \int_{\partial K - \beta} e^{\mu}\,\eta\cdot Y + H \int_K e^{\mu}\,\nu \cdot Y
		\end{equation} 
		yields a well-defined homomorphism
		\[
			\ff_{B}:\rc(N, \Sigma, B)\otimes\lgs\to \R\ .
		\]
		Here $\,Y\,$ is any Killing field in $\,\lgs$,  $\,K\,$ is any cap in $\,\c$, 
		and $\beta\in B$ is the spine of $\,\c\,$ given by Definition \ref{def:assocspine}.
	\end{thm}
	
	\begin{proof}[Proof of Theorem \ref{thm:law}] 
		The basic linearity properties of the integral make $\,\ff_{B}\,$ a homomorphism
		once we establish well-definition: that $\,\ff_{B}[\c](Y)\,$ doesn't
		depend on which cap $\,K\in\c\,$ we use to compute it. We thus need to show
		for all $\,Y\in \lgs$, and all $\,K,K'\in\c$, that
		\begin{equation}\label{eqn:wd}
			\int_{\partial K-\beta} e^{\mu}\,\eta\cdot Y + H \int_K e^{\mu}\,\nu \cdot Y
			=
			\int_{\partial K'-\beta}e^{\mu}\,\eta\cdot Y + H \int_{K'} e^{\mu}\,\nu \cdot Y
		\end{equation}
		for any other $\,K'\in\c$. This follows easily from the constrained first-variation 
		formula (\ref{eqn:vcfvf2}), however. 
		
		For $\,\mu\,$ is $\gs$-invariant, and $\,Y\,$ generates a flow
		that leaves both $\mv{S}$ and $\mv{U}$ unchanged, and hence the left-hand side of 
		(\ref{eqn:vcfvf2}) must vanish.
		The integral over $\,S\,$ on the right of (\ref{eqn:vcfvf2}) vanishes too,
		because $\,h_{\mu}\equiv H$.  So (\ref{eqn:vcfvf2}) reduces to
		\[
			0
			=
			\ds{\ 
			\int_{\partial K - \partial K'} e^{\mu}\,{\eta} \cdot Y 
				+ H\int_{K-K'} e^{\mu}\,{\nu} \cdot Y
			}
		\]
		
		This is clearly equivalent to (\ref{eqn:wd}), since the integrals over 
		$\beta$ there cancel.
	\end{proof}

	\begin{rem}
		The simplest case of Theorem \ref{thm:law}, 
		where $\,\mu\equiv 0\,$ and $\,\gs$ is the full isometry group of $\,N\,$ (so that $\,\lgs\,$ 
		includes all Killing fields) already improves on the conservation law in \cite{kks} by 
		eliminating the triviality assumptions there on $\,H_{n-1}(N)\,$ and $\,H_{n-2}(N)$. 
	\eob
	\end{rem}
			
	\begin{rem}
		The particular choice of spine $\,B\,$ in Theorem \ref{thm:law} is of no real consequence. 
		For when $\,B\,$ and $\,B'\,$ are both spines for $\,(N,\Sigma)$, the well-definition of 
		$\,\ff_{B}\,$ on a class in $\rc(N, \Sigma, B)$ implies that of $\,\ff_{B'}\,$
		on a corresponding class in $\rc(N,\,\Sigma, B')$.
		
		To see this, suppose $\,\phi_{B}\,$ is well-defined on a class $\c$ containing a cap
		$\,K\,$ with boundary $\Gamma + \beta$, where $\beta \in B$ and $\Gamma$ is supported in $\Sigma$. 
		Then there exists a cycle $\beta' \in B'$ homologous to $\beta$, and hence
		an $(n-1)$ chain $\,P\,$ with 
		\[
			\bd P = \b'-\b\ .
		\]
		We claim $\,\phi_{B'}\,$ will now be well-defined on the class $\c'$ represented by $\,K+P\,$
		in $\rc(N, \Sigma, B')$. 
		
		Indeed, take any cap $\,\tilde K\,$ homologous to $\,K+P\,$ in the latter group. 
		Then $\,\tilde K-P\in\c\in \rc(N, \Sigma, B)$, and if $\,\phi_{B}\,$ is well-defined there for some 
		$\,Y\in\lgn$, we have, on the one hand,
		
		\[
			\phi_{B}(\tilde K-P, Y)=\phi_{B}(K,Y)
		\]
		
		On the other hand, we have
		\begin{eqnarray*}
			\phi_{B}(\tilde K-P,Y)
			&=&
			\int_{\G'}e^{\mu}\,\eta\cdot Y + H\int_{\tilde K-P}e^{\mu}\,\nu\cdot Y\\
			&=&
			\int_{\G'}e^{\mu}\,\eta\cdot Y + H\int_{\tilde K}e^{\mu}\,\nu\cdot Y 
				+ H\int_{P}e^{\mu}\,\nu\cdot Y\\
			&=&
			\phi_{B'}(\tilde K,Y) + H\int_{P}e^{\mu}\,\nu\cdot Y
		\end{eqnarray*}
		Together, these facts yield
		\[
			\phi_{B'}(\tilde K,Y) = \phi_{B}(K,Y)-H\int_{P}e^{\mu}\,\nu\cdot Y
		\]
		Since $\,\tilde K\,$ was arbitrary in $\,\c'$, while $\,P\,$ is fixed, we see that
		$\,\phi_{B'}\,$ is well-defined on $\,\c'\in\rc(N, \Sigma, B')$, as claimed.
	\eob
	\end{rem}	
	\bigskip

\section{Partial converse}\label{sec:converse}
	
	Suppose the isometry group $\,\gn\,$ of our ambient manifold $\,N\,$ contains
	a closed, connected group $\,\gs\,$ preserving a density function $\,e^{\mu}\,$ as above.
	Consider an immersed hypersurface $\,f:\Sigma\to N$, together with a spine 
	$B$ for the pair $\,(N,\Sigma)$.  
	
	Above, we assumed constancy of $\mu$-mean curvature on $\,\Sigma$, and deduced conservation of
	of flux. We now seek a \emph{converse} conservation law to the effect that well-definition of
	the flux functional $\,\ff_{B}\,$ implies constancy of $\mu$-mean curvature.
	Well-definition of $\ff_{B}$, however, means nothing without Killing fields on which to pose it, 
	so the strength of any such converse must correlate with the abundance of Killing fields.
	
	Similarly, one shouldn't need to assume well-definition of $\,\ff_{B}\,$ on \emph{all} Killing 
	fields to get a conservation law.
	We could restrict $\,\ff_{B}\,$ to a non-empty subset of $\,\lgs\,$ (even a singleton) and
	ask whether well-definition of $\ff_{B}$ there influences geometry.
	
	Dually, we needn't assume constancy of $\,\ff_{B}\,$ on all caps.
	We have in mind the case where $\,\Sigma\,$ is preserved by a closed, connected subgroup 
	$\,\gg\subset\gs\,$ and $\,\ff_{B}\,$ takes a fixed value on a sufficiently ``crowded'' set of 
	homologous $\,\gg\,$-invariant caps. 
	
	\begin{definition}[\ggc]\label{def:crowded}
		A set of trivial caps $\cs \subset \mathcal B_{n-1}(N, \Sigma)$ is a \textbf{$\gg$-crowded set 
		of boundaries} if, for every $\gg$-orbit $\lambda$ and every $\epsilon > 0$, we can find a cap 
		$K \in C$ satisfying
		\[
			K = \partial U - S
		\]
		Here $U$ is an $(n+1)$-chain in $N$, and $S$ is an $n$-chain in $\Sigma$, which, as an
		$n$-current, is represented by some constant multiple of a submanifold-with-boundary within 
		distance $\epsilon$ of the orbit $\lambda$. In other words, for some constant $\,c\,$
		and any integrable function $f$, we have\smallskip
		\[
			\int_S f = c \int_{\text{spt}\,S} f
		\]
		
		We say that a set $\cs$ of \emph{non}-bounding caps in $\mathcal Z(N, \Sigma, B)$ 
		is \textbf{\ggc} if the difference set $\,\{K-K'\colon\ K,K'\in\cs\}\,$ forms
		a \ggc\ set of boundaries.  Note that in this case, each $\,K\in\cs\,$ represents
		the same reduced cap in $\rc(N, \Sigma, B)$.
	\eob
	\end{definition}
	 
	\begin{eg}
		For any point $p \in \Sigma$, and $\epsilon > 0$, let $V_\epsilon(p)$ be the $\gg$-orbit of 
		the ball $B_\epsilon(p)$. When $\epsilon$ is sufficiently small, $\Sigma$ will separate 
		$V_\epsilon$ into two open sets $V_\epsilon^{\pm}$.  Then $K = \partial V_\epsilon^+ - \Sigma 
		\cap V_\epsilon$ will be a trivial cap, and the collection of all these $K$ for $p \in \Sigma$ 
		and $\epsilon > 0$ small will form a $\gg$-crowded set.
	\end{eg}

	Using this definition, we can state and prove our partial converse, which says (roughly) 
	that when our hypersurface $\,\Sigma\,$ and the density $\,e^{\mu}$ are preserved by a 
	closed, connected subgroup $\,\gg\subset\gn$,
	and the flux is constant on a \ggc\ set of caps --- with respect to Killing 
	fields that commute with $\,\gg\,$ --- we can split $\,\Sigma\,$ into two nice subsets:  
	One with constant $\mu$-mean curvature, and one preserved by the flows of those Killing 
	fields. These subsets may overlap, and either can be empty as seen in Examples \ref{eg1} 
	below.
	
	\begin{thm}\label{thm:converse}
		Let $\,\Sigma\subset N\,$ be a complete oriented $\gg$-invariant hypersurface, and $\,B\,$
		a spine for the pair $\,(N,\Sigma)$. 
		Suppose $\,\cs\subset \mathcal Z(N, \Sigma, B)\,$ is a \ggc\ set of caps, and $\beta\in B$ 
		is the spine of the reduced cap containing $\,\cs$.

		If $\,\gg\,$ preserves a Killing field $\,Y$, and the $\mu$-weighted flux functional 
		\[
			\int_{\partial K - \beta}e^{\mu}\,\eta\cdot Y + H\int_{K}e^{\mu}\,\nu\cdot Y
		\]
		
		is constant on $\,\cs$, then the set
		\[
			\Sigma':=\Sigma\setminus h_{\mu}^{-1}(H)
		\]
		is preserved by the flow of $\,Y$.
	\end{thm}

	\begin{proof}
		Definition \ref{def:crowded} and the form of the flux functional immediately show that
		constancy of flux on any \ggc\ set of {caps} in $\,\mathcal Z(N, \Sigma, B)\,$
		forces \emph{vanishing} of flux on a \ggc\ set of \emph{boundaries}. 
		So without losing generality, we may assume $\cs\subset\bo_{n-1}(N,\Sigma)$.
		
		The heart of our argument then lies with the following
		
		\textbf{Claim:} 
		\emph{If $\,p\in\Sigma'$, then $\,Y_{p}\in T_{p}\Sigma$. }
		
		The definition makes $\Sigma'$ relatively open in $\Sigma$.  Since $\gg$ preserves $\Sigma$ and 
		$\mu$, it preserves $h_\mu$ and hence $\Sigma'$.  The $\gg$-crowdedness of $\cs$ ensures the 
		existence of a cap
		\[
			K = \partial U - S \in \cs,
		\]
		with $S \subset \Sigma'$ supported within an arbitrarily small distance to 
		the $\gg$-orbit of $p$. 
				
		We use the volume-constrained first-variation formula with $K$ as above, and $K' = 0$ since 
		$K$ bounds modulo $\Sigma' \subset \Sigma$.  The first two integrals in (\ref{eqn:vcfvf2}) 
		now vanish on our Killing field $\,Y$, since together they compute the flux of $\,Y\,$ across 
		a trivial cap. 
		
		This reduces the constrained first-variation to a single integral:\smallskip
		\begin{equation*}
			 \delta_{Y}\left(\mv{S}-H\mv{U}\right) 
			 = 
			 \int_{S} e^{\mu}\left(h_{\mu}-H\right)Y\cdot \nu
		\end{equation*} 
		
		Finally, since $\,Y\,$ preserves $\mu$, the left side of this equation must 
		\emph{vanish}, leaving the identity\smallskip
		\begin{equation}\label{eqn:h-H}
			\int_{\text{spt}\,S} e^{\mu}\left(h_{\mu}-H\right)Y\cdot \nu = 0\ .
		\end{equation}

		If $Y_p \not\in T_pS$, then by assumption the integrand $(h_\mu - H) Y\cdot \nu \neq 0$ at $p$.  
		Since all quantities are continuous and preserved by $\gg$, it follows that 
		$(h_\mu - H) Y\cdot \nu$ is strictly positive (or negative) in a neighborhood of the 
		$\gg$-orbit of $p$.  Since the $\gg$-crowdedness of $\cs$ lets us confine the support of 
		$\,S\,$ to such a neighborhood, we can contradict equation \eqref{eqn:h-H}, thereby 
		proving the claim.
		
		To finish proving the Theorem, it suffices to show that whenever $\,p\in\Sigma'$,
		the entire $Y$-streamline with initial velocity $\,Y_{p}\,$ lies in $\,\Sigma'$.
		
		Let $\,T>0\,$ be the maximal time such that $\,\Th(p,t)\subset\Sigma'\,$ for \textit{all}
		$\,t<T$. By Lemma \ref{lem:nbd} (with $\,q:=p$), some such $\,T\,$ exists.
		Since $\,Y\,$ generates a $\mu$-preserving \textit{isometric} flow, we have 
		$\,h_{\mu}(\Th(p,t))\equiv H'\,$ with $\,H'\,$ constant for all $\,t\in[0,T)$. Moreover,
		$\,H'\ne H$, as we are in $\,\Sigma'$. 
		We now claim $\,T=\infty$. For otherwise, the continuity of $\,h_{\mu}\,$ and the completeness
		of the larger hypersurface $\,\Sigma\,$ immediately yields both $\,\Th(p,T)\in\Sigma$, and 
		$\,h_{\mu}(\Th(p, T))=H'\ne H$, so that $\,\Th(p,T)\in\Sigma'$. But then Lemma \ref{lem:nbd} 
		(with $\,q:=\Th(p,T)$) contradicts the maximality of $\,T\,$. In short, $\,\Th(p,t)\in\Sigma'\,$
		for \textit{all} $\,t\ge 0\,$. Since the same reasoning shows that $\,\Th(p,t)\in\Sigma'\,$
		for all $\,t\le 0\,$ too, the proof is complete.
	\end{proof}
		
	\begin{rem}\label{rem:triv} 
		We emphasize again that our converse remains interesting even when $\,\gg\,$
		is trivial. Theorem \ref{thm:converse} then implies, for instance, that when the 
		flux across every sufficiently small trivial 
		cap vanishes on the generators of a subgroup $\,\gs\subset\gn$, the 
		part of $\,\Sigma\,$ that does \emph{not} have constant $\mu$-mean curvature 
		$\,h_{\mu}= H\,$ must be $\gs$-invariant. 
	\eob
	\end{rem}
	
	\begin{cor}
		If, as in Theorem \ref{thm:converse}, the $\mu$-weighted flux functional is constant on one 
		\ggc\ set of caps, it actually extends as a well-defined conserved quantity to all of 
		$\rc(N, \Sigma, B)$.
	\end{cor}
	
	\begin{proof}
		While the Theorem assumes constancy of $\,\ff_{B}\,$ only on a 
		\ggc\ set of caps, the proof then deduces that at every point $\,p\in\Sigma$,
		either $\,h_{\mu}=H$, or $\,Y\,$ belongs to $\,T_{p}\Sigma$. 
		In this case, the last integral in the volume constrained first-variation formula 
		(\ref{eqn:vcfvf}) clearly vanishes on any $(n-1)$-chain $S$ in $\,\Sigma$, so that 
		$\,\ff_{B}(K,Y)=\ff_{B}(K',Y)\,$ for any two homologous caps 
		$K, K'\in \mathcal Z(N, \Sigma, B)$.
	\end{proof}

	Let us henceforth agree that when $\,\gg\,$ is trivial, we call a \ggc\ set of caps simply
	\textbf{crowded}.
	
	\begin{cor}
		If $\,N\,$ is homogeneous, $\mu$ is constant, and on some crowded set of caps,
		the flux functional is well-defined for
		all Killing fields on $\,N$, then $\,\Sigma\,$ has mean curvature $\,h\equiv H\,$
		everywhere.
	\end{cor}
	
	\begin{proof}
		With $\,\gg\,$ trivial in Theorem \ref{thm:converse},
		well-definition on all Killing fields makes $\,\Sigma'\,$ invariant
		under the entire isometry group $\,\gn$. But in a homogeneous space, all non-empty
		$\,\gn\,$-invariant sets have top dimension. So $\,\Sigma',$ having codimension one, 
		must be empty, forcing $\,h\equiv H\,$ throughout $\,\Sigma$. 
	\end{proof}
	
	When $\,\gs\subset\gn\,$ is a subgroup, we say that $\,N\,$ has  \textbf{cohomogeneity} $\,k\,$ 
	with respect to $\,\gs$ when the highest dimensional orbits of $\,\gs\,$ have codimension $k\,$ 
	in $\,N$. Cohomogeneity {zero} is the same as \emph{homogeneity}.
	
	\begin{cor}\label{cor:coho1}
		Suppose a real-analytic riemannian manifold $\,N\,$ has cohomogeneity one with respect to a
		$\mu$-preserving group $\gs$, and 
		on some crowded set of caps, the flux functional is well-defined on all of 
		$\,\lgs$.  Then either $\,h_{\mu}\equiv H$, or else $\,\Sigma\,$ is an orbit of $\,\gs$.
		Either way, $\,h_{\mu}\,$ is constant on $\,\Sigma$.
	\end{cor}
	
	\begin{proof}
		In an analytic ambient space, hypersurfaces with constant $\mu$-mean curvature are analytic
		\cite[5.2.16]{gmt}.
		Cohomgeneity one means the only connected $\gs$-invariant hypersurfaces are single orbits of 
		$\,\gs$, which clearly have constant $\mu$-mean curvature.
		Since $\,\Sigma\,$ is connected, the Corollary now follows from Theorem 
		\ref{thm:converse}.
	\end{proof}
	
	\subsection{Examples.}\label{eg1}
		Take $\,N=\R^{3}$, let $\gg\,$ be the circular group acting by rotation
		about the $x$-axis. The Killing field $Y=(1,0,0)$ generates translational flow along that
		axis, and $\gg$ commutes with this flow as required by Theorem \ref{thm:converse}. 
		The non-cylindrical Delaunay surfaces---\cmc\ 
		surfaces of revolution about the $x$-axis analyzed by C.~Delaunay in 1841---show that 
		Theorem \ref{thm:converse} may obtain with $\gg$-invariant hypersurfaces having 
		$\,h\equiv H\,$ and \textit{no} flow-invariant subset $\,\Sigma'$.
		
		Contrastingly, if we take $\,\Sigma\,$ to be any cylinder centered about the $x$-axis with 
		radius \emph{not} equal to $\,1/H$, we get an example with $\,\Sigma'=\Sigma$. 
		That is, $\,\Sigma\,$ has mean curvature $\,H\,$ {nowhere}, and yet the flux functional
		remains well-defined on $\,Y$, thanks to the global flow-invariance of $\,\Sigma$.
		
		Of course, the cylinder of radius $\,1/H\,$ about the $x$-axis has \emph{both} 
		$\,h\equiv H\,$ \emph{and} the extra translational symmetry.
		
		All these possibilities arise in the family of twizzlers too, as
		we shall shortly see.

	\subsection{Case study.} \label{eg:twiz3}(First integrals for twizzlers.)
		Consider the riemannian product $\,N:=\C\times\US^{1}_{R}$, where the complex plane $\,\C\,$ 
		and $\,\US^{1}_{R}\,$ (the circle  of radius $R$) have their standard metrics.
		Take $\,\mu\equiv 0$, and let $\,\gg\approx\US^{1}\,$ 
		act via screw-motion:
		\[
			[e^{i t}]\left(z,\,Re^{i\th}\right) = \left(e^{i t}z,\, Re^{i(t+\th)}\right)
		\]
		In this situation, each helical orbit of the $\gg$-action generates $\,H_{1}(N)\approx\Z$.  
		Let $\,\Sigma\subset N\,$ be any connected $\gg$-invariant surface, and with no loss of generality assume it does not contain
		the shortest orbit $\,\b:=\mathbf{0}\times\US^{1}_{R}$. Then $\,\b\,$ clearly generates
		a spine for $\,(N,\Sigma).$
		
		We can parametrize $\,\Sigma\,$ by letting $\,\gg\,$ act on an immersed 
		curve $\,\g:\R\to\C\times\{1\}\approx\C\,$ via the map 
		\begin{equation}\label{eq:X}
			X(u,v) = \left(e^{\i v}\g(u),\,Re^{\i v}\right).
		\end{equation}
		Assume the orientation
		of $\,\g\,$ makes the natural frame $\,\{X_{u},X_{v}\}\,$ \emph{positively} oriented
		along $\,\Sigma$.

		Now fix any point $\,p\,$ on the generating curve $\,\g$, and join it to 
		the origin in $\,\C\,$ by a line segment. This segment sweeps out a helicoidal 
		{cap} $K^{p}$, invariant under the $\gg$-action, and the reduced class of $K^{p}$
		in $\,\rc(N,\Sigma, B)\,$ is clearly independent of $\,p$.  One easily sees
		that as $p$ varies over $\,\g$, the resulting caps $\,K^{p}\,$ form a 
		{\ggc} set $\,\cs\,$ according to Definition \ref{def:crowded}.

		Now let $\,Y\,$ be the circular Killing field generating the purely ``horizontal'' 
		isometric flow $\,[e^{\i s}](z,Re^{\i\th}) = (e^{\i s}z,Re^{\i\th})$. 
		Note that $\,Y\,$ commutes with $\,\gg\,$ and preserves 
		$\,\mu$, as required by Theorem \ref{thm:converse}. 
		
		Finally, suppose that when we put $\,K=K^{p}\,$ and $\,\b\,$ as above in the
		the flux formula of Theorem \ref{thm:law}, the result is independent of $\,p$.
		
		Since $\,N\,$ has cohomgeneity one with respect to the extension of $\,\gg\,$ by
		the flow of $\,Y$, Corollary \ref{cor:coho1} dictates that \emph{either} $\,\Sigma\,$ 
		is a \cmc\ twizzler with $\,h\equiv H$, or else $\,\Sigma\,$ is an orbit of the 
		combined action, and thus a circular cylinder with $\,h\equiv 1/r\,$ 
		($r$ giving the radius of the cylinder; typically $\,1/r\ne H$).

		As an application of our theory, we now show that constancy of $\,\ff_{B}\,$ on the \ggc\
		set of caps $\,K^{p}\,$ described above ``explains'' the first-order ODE
		known to characterize generating curves of \cmc\ twizzlers, as mentioned in our introduction. 
		
		\begin{prop}\label{prop:fint}
			A non-circular immersed curve $\,\g\,$ in $\,\C\,$ generates a 
			twizzler in $\,\C\times \US^{1}_{R}\,$ with $\,h\equiv H\,$ if and only if for
			some $\,c\in\R$, it solves
		
			\begin{equation}\label{eqn:1rstInt}
				\frac{2\pi\,R^{2}\,(\dot\g\cdot\i\,\g)}
				     {\sqrt{R^{2}|\dot\g|^{2}+(\dot\g\cdot\g)^{2}}}
				- \pi\,R\,H\,|\g|^{2}=c
			\end{equation} 
		\end{prop}
		
		\begin{proof}
			Since we assume $\,\g\,$ is not circular, Theorem \ref{thm:law} and Corollary \ref{cor:coho1}, 
			as noted above, tell us that $\,h\equiv H\,$ if and only if the flux of the circular 
			vectorfield\smallskip
			\[
				Y_{(z,\tau)} = -\left(\i z, 0\right)
			\]
			
			across $\,K^{p}\,$ is independent of $\,p$. That is,\smallskip
			\begin{equation}\label{eqn:phi=c}
				\ff_{B}\l(K^{p},Y\r) \equiv c\quad\text{for all $\,p\in\g\,$}.
			\end{equation}
			
			Equation (\ref{eqn:1rstInt}) merely evaluates this assertion.
						
			To reach (\ref{eqn:1rstInt}) from (\ref{eqn:phi=c}), we temporarily fix a point
			$\,p=\g(t)\,$ on the generating curve $\,\g$, and specify an orientation on 
			the  cap $\,K^{p}$, by declaring the frame field 
			$\,\{K_{u},K_{v}\}\,$ associated with the parametrization
			\begin{equation*}
				K(u,v) 
				=
				\left(u\,e^{iv}p,\, R\,e^{iv}\right),\quad
				(u,v)\in\left(0,1\right)\times\left(0,2\pi\right),
			\end{equation*} 
			to be {positively} oriented.
			
			Now consider the second integral in the flux formula (\ref{eqn:flux})---the one that 
			pairs $\,Y\,$ with the unit normal $\,\nu\,$ along $\,K^{p}$.  The correctly oriented
			unit normal will be a positive multiple of
			\[
				K_{u}\wedge K_{v} = \left(-R\,\i\,e^{iv}p,\, u\,|p|^{2}\,\i\,e^{iv}\right).
			\]
			The length of $K_{u}\wedge K_{v}$ is actually irrelevant: we divide by it to 
			normalize, but then multiply it back in as the Jacobian in the flux integral, namely
			\begin{equation*}\label{eqn:2ndFI}
				\int_{K^{p}} \nu\cdot Y 
				\ =\ 
				\int_{0}^{2\pi}\int_{0}^{1} 
				\left(K_{u}\wedge K_{v}\right)\cdot Y\big|_{K(u,v)}
				\ du\,dv\\
			\end{equation*} 
			
			At $\,K(u,v)$, we have $\,Y = -\left(u\,\i\,e^{\i v}p\,,\,0\right)$, so
			the corresponding flux term evaluates easily to
			\begin{equation}\label{eqn:flx2}
				H\int_{K^{p}} \nu\cdot Y 
				\ =\ 
				2\pi\,HR\,|p|^{2}\int_{0}^{1} u\ du \ =\ \pi\,HR\,|p|^{2}
			\end{equation} 
			
			Now consider the other integral in the flux formula (\ref{eqn:flux}), the
			integral over $\,\Gamma := \partial K - \beta$, where $\,K^{p}\,$ 
			meets $\Sigma$. This curve is the helical $\gg$-orbit of $\,p$, 
			and one easily computes its length as
			\[
				\left|\G\right| = 2\pi\sqrt{R^{2}+|p|^{2}}\ .
			\]
			Our chosen orientation of $K$ induces an orientation on $\Gamma$.  Since $K_u$ at $\Gamma$ is parallel to the outer conormal in $K$, the velocity $\Gamma'$ of $\Gamma$ is equal to a positive multiple of $X_v$.  
			The outer conormal in $\Sigma$ along $\G$, which we called $\,\eta$, must then give the pair $\,\{\eta,\G'\}\,$ positive orientation, so
			we can obtain $\,\eta\,$ by orthonormalizing $\,X_{u}\,$ along $\,\G$, i.e., 
			by normalizing
			\[
				|X_{v}|^{2}X_{u}+\left(X_{u}\cdot X_{v}\right) X_{v}\ .
			\]
			Both $\eta$ and $Y$ are $\gg$-invariant, making $\,\eta\cdot Y\,$ constant along
			$\,\G$, and careful calculation then shows that indeed,
			\[
				\eta\cdot Y 
				\equiv 
				\frac{-R^{2}\,\dot\g\cdot\i\,p}
				     {\sqrt{R^{2}+|p|^{2}}\sqrt{(R^{2}+|p|^{2})|\dot\g|^{2}-(\g\cdot\i\,p)^{2}}}
			\]
			where we evaluate $\,\dot\g\,$ at $\,p$.  We can simplify the second square root
			in the denominator here via the elementary identity 
			\[
				\left(\dot\g\cdot\i\,p\right)^{2}
				=
				\left|\dot\g\right|^{2}\left|p\right|^{2}-\left(\dot\g\cdot p\right)^{2}
			\]
			This lets us express the conormal flux integral as
			
			\begin{equation}\label{eqn:flx1}
				\int_{\G}\eta\cdot Y = 
				-\frac{2\pi\,R^{2}\,(\dot\g\cdot\i\,p)}
				     {\sqrt{R^{2}|\dot\g|^{2}+(\dot\g\cdot p)^{2}}}
			\end{equation}
			Setting $\,p=\g(t)\,$ and recalling (\ref{eqn:flux}), we now get $\,\ff_{B}(K^{p},Y)\,$ by 
			adding (\ref{eqn:flx2}) to (\ref{eqn:flx1}).
		\end{proof}
	
	\begin{rem}
		If we parametrize a convex arc of the generating curve $\,\g\,$ using its 
		\textbf{support function}, namely
		\[
			k(t) := \sup_{\th}\g(t)\cdot e^{\i\th}
		\]
		then
		\[
			\g(t) = \left(k(t) + \i\,\dot k(t)\right)e^{\i\,t}
		\]
		
		It now follows from Proposition \ref{prop:fint} that when $\,\g\,$ generates a
		pitch-$R$ twizzler with $\,h\equiv H$, its support function satisfies a
		simple non-linear ODE:
		\[
		\frac{2R\,k}{\sqrt{R^{2}+\dot k^{2}}}-H\left(k^{2}+\dot k^{2}\right)=C
		\]
		
		In other words, the phase portrait of $\,k\,$ lies on one of the ``heart-shaped'' level
		curves of the function
		\[
			F(x,y) := \frac{2R\,x}{\sqrt{R^{2}+y^{2}}} - H\left(x^{2}+y^{2}\right)
		\]
		In \cite{p1} and \cite{p2}, Perdomo based his dynamical characterization of twizzler 
		generating curves and his study of their moduli space, on this observation.
	\eob
	\end{rem}
	
	\begin{rem}[Twizzlers in other 3D-space forms]\label{rem:etc}
		It is natural to see the curve $\,\g\,$ in our case study \ref{eg:twiz3} above as the 
		projection of the hypersurface $\,\Sigma\,$ into the \emph{orbit space} $\,N/\gg\approx \C$. 
		The length of the orbit above $\,z\in\C\,$ is easily computed as 
		$\,|\G_{z}|=2\pi\sqrt{R^{2}+|z|^{2}}$, and if we adopt this as our density function, i.e.,  
		$\,e^{\mu(z)}=|\G_{z}|\,$, on the orbit space (cf.~Definition \ref{def:mass}), 
		a simple reworking of Proposition \ref{prop:fint} re-interprets the first integral there as 
		the condition for $\,\g\,$ to have $\,h_{\mu}\equiv H\,$ as a ``hypersurface'' in the 
		two-dimensional orbit space.
		
		Similarly, one can seek \cmc\ ``twizzlers'' in the 3-sphere $\,\US^{3}\subset\R^{4}\,$
		invariant under one of the helical $\,(k,l)\,$ ``torus knot'' circle
		actions given by 
		\[
			[e^{\i t}]\left(z,w\right) = \left(e^{\i kt}z,\,e^{\i lt}w\right)
		\]
		This is the standard Hopf action when $\,k=l=1$, in which case the orbit space 
		$\,\US^{3}/\gg\,$ is of course the standard 2-sphere $\,\US^{2}$.  More generally,
		when $\mathrm{gcd}(k,l)=1$, on can realize the orbit space as an eccentric ``football''
		or ``teardrop''shaped surface of revolution in $\,\R^{3}$, smooth except for conical singularities
		at one or both ends. The $\gg$-invariant \cmc\ twizzlers in $\,\US^{3}\,$ then correspond
		one-to-one with curves having constant $\mu$-mean curvature in the the orbit space, where
		the density function is again given by orbit-length: $\,e^{\mu(p)}=|\pi^{-1}(p)|\,$ for 
		$\,p\,$ in the orbit space. By Theorem \ref{thm:converse}, these $\,h_{\mu}\equiv H\,$ curves 
		are the precisely the non-circular curves that conserve flux
		along the Killing fields that generate the rotational symmetry of the orbit space. 
		It is then straightforward to use this fact, as in Proposition \ref{prop:fint}, to derive
		the first integral they satisfy. See \cite{e} for the resulting expression. We should
		note here that the special case
		$\,h_{\mu}\equiv 0\,$ (\textit{minimal} twizzlers in $\,S^{3}$) was analyzed using Hamilton-Jacobi
		theory in \cite[Chap. IV]{hl}.
		
		Analogous helical actions exist in the hyperbolic space form $\mathbb{H}^{3}$, and
		the resulting \cmc\ twizzlers have a first integral derivable in precisely
		the same way. The reader may consult \cite{e} for a description of the group action 
		and the resulting first integral in this case as well.
	\eob
	\end{rem}
	
%%%%%%%%%%%%%%%  ACKNOWLEDGMENTS  %%%%%%%%%%%%%%%%%

\section*{Acknowledgments} 
The Summer 2011 Math REU program at Indiana University, Bloomington initiated and supported this
work. We gratefully acknowledge the National Science Foundation for funding that program.

%%%%%%%%%%%%%%%%%  BIBLIOGRAPHY %%%%%%%%%%%%%%%%%%%

\end{document}